\documentclass[conference,a4paper,10pt]{IEEEtran}
\usepackage{amssymb}
\usepackage{amsfonts}
\usepackage{amsmath}
\usepackage{amsthm}
\usepackage{graphicx}
\usepackage{epsfig}
\usepackage{psfrag}
\usepackage{color}
\usepackage{dsfont}
\makeatletter
\xdef\@endgadget#1{{\unskip\nobreak\hfil\penalty50\hskip1em\hbox{}\nobreak
    \hfil#1\parfillskip=0pt\finalhyphendemerits=0\par}}
\def\@qedsymbol{${}_\blacksquare$}
\def\qed{\@endgadget{\@qedsymbol}}

\newtheorem{theorem}{Theorem}

\newtheorem{definition}{Definition}

\newtheorem{remark}{Remark}
\newcommand{\mR}{\mathbb{R}}

\newcommand{\X}{\mathcal{X}}
\newcommand{\cS}{\mathcal{S}}

\newcommand{\bq}{\begin{equation}}
\newcommand{\eq}{\end{equation}}

\usepackage{tikz}
\usepackage{circuitikz}
\usepackage{tkz-berge}
\usetikzlibrary{calc}
\usetikzlibrary{shapes,fit}
\usepackage{arydshln}                                   
\usetikzlibrary{shapes,arrows}
\usetikzlibrary{positioning}

\def\BibTeX{{\rm B\kern-.05em{\sc i\kern-.025em b}\kern-.08em
    T\kern-.1667em\lower.7ex\hbox{E}\kern-.125emX}}

\usepackage{xcolor}

\usepackage{newtxtext}
\usepackage{newtxmath}

\begin{document}

\title{\bfseries Limits to Energy Conversion}

\author{\IEEEauthorblockN{Arjan van der Schaft}
\IEEEauthorblockA{Bernoulli Institute for Mathematics, Computer Science and AI\\ 
University of Groningen\\ 
PO Box 407, 9700 AK, The Netherlands\\ 
E--mail: a.j.van.der.schaft@rug.nl}
\and
\IEEEauthorblockN{Dimitri Jeltsema} 
\IEEEauthorblockA{School of Engineering and Automotive\\ 
HAN University of Applied Science,\\ 
P.O.~Box 2217, 6802 CE Arnhem, The Netherlands\\  
E--mail: d.jeltsema@han.nl}
}

\maketitle

\begin{abstract}
A consequence of the Second Law of thermodynamics is that no thermodynamic system with a single heat source at constant temperature can convert heat into mechanical work in a recurrent manner. First we note that this is equivalent to cyclo-passivity at the mechanical port of the thermodynamic system, while the temperature at the thermal port of the system is kept constant. This leads to the general system-theoretic question which systems with two power ports have similar behavior: when is a system cyclo-passive at one of its ports, while the output variable at the other port (such as the temperature in the thermodynamic case) is kept constant? This property is called `one-port cyclo-passivity', and entails, whenever it holds, a fundamental limitation to energy transfer from one port (where the output is kept constant) to the other port. Sufficient conditions for one-port cyclo-passivity are derived for systems formulated in general port-Hamiltonian form. This is illustrated by a variety of examples from different {(multi-)}physical domains; from coupled inductors and capacitor microphones to synchronous machines.
\end{abstract}

\section{INTRODUCTION}

Energy conversion is a common phenomenon in many multi-physics systems: electro-mechanical, electro-chemical, thermal-mechanical, electro-kinetic, thermal diffusion, etc.. The Second Law of thermodynamics implies that no thermodynamic system with a single heat source at constant temperature can convert heat (thermal energy) into work (mechanical energy) in a recurrent manner. This gives rise to the natural question: 
\begin{center}
\emph{Are the limitations imposed by the Second Law of thermodynamics on energy conversion in thermodynamic systems also present in other cases?}
\end{center}
In this paper, we initiate a general system-theoretic treatment of this question along the following lines. Consider a cyclo-passive system with two power ports $(u_1,y_1)$ and $(u_2,y_2)$ schematically depicted in Fig.~\ref{fig:two-port}. 

\begin{figure}[h]
\begin{center}
\psfrag{u}[][]{$u_1$}
\psfrag{y}[][]{$y_1$}
\psfrag{v}[][]{$u_2$}
\psfrag{w}[][]{$y_2$}
\includegraphics[width=0.6\columnwidth]{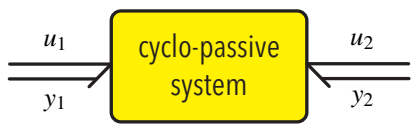}
\caption{Two-port cyclo-passive system.}
\label{fig:two-port}
\end{center}
\end{figure}

\noindent When is the system, constrained by imposing $y_1$ to be constant, cyclo-passive at the second port? We provide sufficient conditions for this to happen, and we illustrate this on a variety of examples from different physical domains. 

The remainder of the paper is organized as follows. In Section \ref{sec:diss_theory}, we start by recalling some well-known properties of dissipative systems and extend them with some recent results on cyclo-dissipativity. Section \ref{sec:thermo_mot} connects the notion of cyclo-dissipativity to the Second Law of thermodynamics and serves as a prime motivator for the main theorem presented in Section \ref{sec:one-port_pass}. Section \ref{sec:examples} illustrates the main result using a variety of examples from different {(multi-)}physical domains. Finally Section \ref{sec:conc} contains conclusions and outlook for further research.

\section{DISSIPATIVITY THEORY}\label{sec:diss_theory}

First we recall the basic definitions of dissipativity theory as originating from the seminal paper \cite{willems1972}, see also \cite{hillmoylan1980, passivitybook}. Consider a system with state vector $x$ and vector of external (e.g., input and output) variables $w$. Consider a scalar-valued {\it supply rate} $s(w)$. Then the system is said to be {\it dissipative} with respect to the supply rate $s$ if there exists a non-negative function $S(x)$ such that along all trajectories of the system and for all $t_1 \leq t_2$ and $x(t_1)$
\bq
\label{diss}
S(x(t_2)) - S(x(t_1)) \leq \int_{t_1}^{t_2} s(w(t)) dt,
\eq
and {\it lossless} if this holds with equality. Equation \eqref{diss} is referred to as the {\it dissipation inequality}. Interpreting $S(x)$ as stored 'energy' while at state $x$, and $s(w(t))$ as `power' supplied to the system at time $t$, this means that increase of the stored energy can only occur due to externally supplied power. Any function $S$ satisfying \eqref{diss} will be called a {\it storage function}. Dissipativity thus means the existence of a non-negative\footnote{Sometimes 'non-negativity' is taken as part of the definition of storage function. However because of the cyclo-dissipativity considerations later on we will deviate from this; see also \cite{vdscyclo}.} storage function. Furthermore, since addition of an arbitrary constant to a storage function again leads to a storage function, the requirement of non-negativity of $S$ can be relaxed to $S$ being {\it bounded from below}.

A variational characterization of dissipativity, in terms of the behavior of the $w$ trajectories (the external behavior of the system), is the following \cite{willems1972}. Define for any $x$ the {\it available storage}
\bq
S_a(x):= \sup_{w, \tau\geq 0} - \int_0^\tau s(w(t)) dt,
\eq
where the supremum is taken over all external trajectories $w(\cdot)$ of the system corresponding to initial condition $x(0)=x$, and all $\tau\geq0$. Obviously, $S_a(x)\geq 0$. Then the system is dissipative with respect to the supply rate $w$ if and only if $S_a(x)$ is finite for every $x$. Furthermore, if  $S_a(x)$ is finite for every $x$ then $S_a$ satisfies \eqref{diss}, and in fact $S_a$ is the {\it minimal non-negative storage function}; see \cite{passivitybook} for further ramifications. Note that, interpreting as above $s(w)$ in \eqref{diss} as 'power' supplied to the system, $S_a(x)$ is the maximal 'energy' that can be extracted from the system at initial condition $x$. Thus the system is dissipative if and only if the maximal 'energy' that can be extracted from the system starting from any initial state is finite.

Dropping the requirement of {\it non-negativity} of the storage function $S$ leads to the notion of {\it cyclo-passivity}, respectively {\it cyclo-losslessness}, cf. \cite{willems1973, hillmoylan1975,passivitybook,vdscyclo}. 




\begin{definition}
A system is {\it cyclo-dissipative} if 
\bq
\label{cyclic}
\int_{t_1}^{t_2}  s(w(t)) dt \geq 0,
\eq
for all $t_2\geq t_1$ and all external trajectories $w$ such that $x(t_2)=x(t_1)$. In case \eqref{cyclic} holds with equality, we speak about cyclo-losslessness.
Furthermore, the system is called cyclo-dissipative {\it with respect to} $x^*$ if \eqref{cyclic} holds for all $t_2\geq t_1$ and all external trajectories $w$ such that $x(t_2)=x(t_1)=x^*$. Finally, in case \eqref{cyclic} holds with equality for all $t_2\geq t_1$ and all external trajectories $w$ such that $x(t_2)=x(t_1)=x^*$, then the system is cyclo-lossless with respect to $x^*$.
\end{definition}
The following theorem in \cite{vdscyclo} extends the results obtained by Hill 
\& Moylan in their technical report \cite{hillmoylan1975} in a number of directions.
\begin{theorem}
\label{cycthm}
If there exists a, possibly indefinite, function $S$ satisfying the dissipation inequality \eqref{cyclic}, then the system is cyclo-dissipative.
Conversely, assume the system is reachable from some (ground) state $x^*$ and controllable to this same state $x^*$. Define the (possibly extended) functions $S_{ac}: \X \to \mR \cup \infty$ and $S_{rc} : \X \to - \infty \cup \mR$ as
\bq
\begin{aligned}
S_{ac}(x) &= \!\!\!\! \mathop{\sup_{w,T\geq 0}}_{x(0)=x,x(T)=x^*} \!\!\!\!\! - \int_0^Ts(w(t)) dt, \\[2mm]
S_{rc}(x) &= \!\!\!\!\!\! \mathop{\inf_{w,T\geq 0}}_{x(-T)=x^*,x(0)=x} \!\! \int_{-T}^0s(w(t)) dt,
\end{aligned}
\eq
where the supremum is taken over all external trajectories $w$.
Then the system is cyclo-dissipative with respect to $x^*$ if and only if
\bq
\label{acrc}
S_{ac}(x) \leq S_{rc}(x), \quad x \in \X.
\eq
In particular, if the system is cyclo-dissipative with respect to $x^*$ then both $S_{ac}$ and $S_{rc}$ are (indefinite) storage functions satisfying \eqref{diss}, and thus the system is cyclo-dissipative.
Furthermore, if the system is cyclo-dissipative with respect to $x^*$ then
\bq
S_{ac}(x^*) = S_{rc}(x^*)=0,
\eq
and any other (indefinite) storage function $S$ satisfies
\bq
\label{**}
S_{ac}(x) \leq S(x) - S(x^*) \leq S_{rc}(x). 
\eq
\end{theorem}
The first statement in this theorem (existence of indefinite storage function implies cyclo-dissipativity) simply follows by substituting $x(t_1)=x(t_2)$ in \eqref{diss}. Furthermore, note that, unlike the dissipativity case, it may be possible for a cyclo-dissipative system to extract an infinite amount of energy from the system (since the storage function may not be bounded from below).

In case of the special {\it passivity} supply rate $s(w)=s(u,y)=y^Tu$, where $w=(u,y)$ and $u$ and $y$ are equally dimensioned vectors, the terminology `dissipativity' in all of the above is replaced by the classical terminology of {\it passivity}. In this case $u$ and $y$ typically are vectors of power-conjugate variables, like forces and velocities, and voltages and currents.

\section{MOTIVATION FROM CLASSICAL THERMODYNAMICS}\label{sec:thermo_mot}

Let us consider a macroscopic thermodynamic system with two external ports\footnote{The discussion can be easily extended to cases where the mechanical port is replaced by, or extended to, other types of ports; e.g., chemical, electrical.}. The second port is the {\it mechanical\footnote{In some cases better called {hydraulic}.} port}, with port variables being the pressure\footnote{In order to stick with the usual notation in thermodynamics we follow the physics convention, where $Pu_V$ is the mechanical work exerted {\it by} the thermodynamic system {\it on} the environment.} $-P$ and the rate of volume change $u_V:=\dot{V}$, where $V$ is the volume. The instantaneous power given by the environment to the thermodynamic system is thus
\bq
-Pu_V = \mbox{ rate of mechanical work}
\eq
The first port is the {\it thermal port}, where the thermodynamic system is connected to a heat source (heat bath), with port variables the temperature $T$ and the heat flow $q$ (heat per second) from the heat source into the system.

The {\it First Law of thermodynamics} is expressed by assuming the existence of a function $E$ of the state $x$ ('total energy') of the thermodynamic system, satisfying along all trajectories
\bq
\begin{aligned}
\label{firstlaw}
E(x(t_2)) - E(x(t_1)) =  \int_{t_1}^{t_2}  q(t) - P(t)u_V(t) dt \\[2mm]
\left(  =  \int_{t_1}^{t_2} q(t) dt - \int_{t_1}^{t_2} P(t)dV(t) \right)
\end{aligned}
\eq
for all $t_1 \leq t_2$.
That is, the increase of the total energy $E$ of the thermodynamic system is equal to the incoming heat (through the thermal port) minus the mechanical work performed by the system on the environment (through the mechanical/hydraulic port). 

Clearly, the First Law of thermodynamics can be equivalently expressed as the cyclo-losslessness of the system with respect to the supply rate $s(q,T,P,u_V)=q - Pu_V$, with storage function $E$. 
Furthermore, in many situations the function $E$ is bounded from below, in which case it can be turned into a non-negative storage function by adding a suitable constant, implying losslessness. Note that for a general thermodynamic system, however, there is no reason why $E$ should be bounded from below.
%
%

The (cyclo-)dissipativity interpretation of the {\it Second Law of thermodynamics} is less clear. The formulation of the Second Law given by Lord Kelvin states that (see e.g. \cite{fermi}):

\smallskip
{\it A transformation of a thermodynamic system whose only final result is to transform into work heat extracted from a source which is at the same temperature throughout is impossible. }

\smallskip

Based on the Second Law, the standard argumentation in classical thermodynamics, see e.g. \cite{fermi}, is to derive, by using the {\it Carnot cycle}, the inequality
\bq
\label{secondlaw}
\oint \frac{q(t)}{T(t)} dt \leq 0
\eq
for all cyclic processes, where equality holds for so-called {\it reversible} cyclic processes. Furthermore, based on this, one defines the {\it entropy} $\cS$ as a function of the state of the thermodynamic system, and derives the {\it Clausius inequality} 
\bq
\label{clausius}
\cS(x(t_2))-\cS(x(t_1)) \geq \int_{t_1}^{t_2} \frac{q(t)}{T(t)} dt.
\eq
This immediately leads to the dissipativity formulation of the Second Law as given in \cite{willems1972}, see also \cite{haddad}. Indeed, \eqref{secondlaw} is the same as saying that the thermodynamic system is {\it cyclo-dissipative} with respect to the supply rate $s(q,T,P,u_V)=-\frac{q}{T}$, while \eqref{clausius} is equivalent to saying that the system is cyclo-dissipative with respect to the supply rate $-\frac{q}{T}$ and storage function $-\cS$. 

In quite a few cases the entropy function $\cS$ is {\it bounded from above}, and thus the storage function $-\cS$ is bounded from below; in which case {\it cyclo}-dissipativity can be replaced by dissipativity. However, similar to the cyclo-passivity or passivity interpretation of the First Law, this is {\it not} always the case.

On the other hand, there is {\it another} (but of course related) cyclo-dissipativity aspect of the Second Law which will serve as a prime motivation for this paper. Indeed, without taking recourse to the definition of the entropy $\cS$ via the Carnot cycle and the Clausius inequality, it is evident that Kelvin's formulation of the Second Law also has the following immediate consequence\footnote{But note that Kelvin's formulation of the Second Law is {\it stronger} than this consequence, since Kelvin's formulation allows for interaction of the system with multiple heat sources of {\it different} temperatures, where however the heat converted into mechanical energy is eventually only extracted from a single heat source.}, which admits a cyclo-dissipativity interpretation.
In fact, since the rate of mechanical work on the system is given by $- Pu_{V}$, Kelvin's formulation of the Second Law {\it implies} that the thermodynamic system with heat source at arbitrary but {\it fixed} temperature $T$ is {\it cyclo-passive} with respect to the power-conjugate variables $-P,u_V$, i.e.,
\bq
 \oint -P(t) u_V(t) dt \geq 0
\eq
along all cyclic processes. This means that in order to convert heat into mechanical work in a recurrent manner, at least two temperatures are needed (such as in the construction of the Carnot cycle between two heat sources of different temperatures). Hence the Second Law has an immediate consequence for the limitations to {energy conversion from the thermal port of the thermodynamic system to its mechanical port}.

Furthermore, if we assume the existence of an (indefinite) storage function (as e.g. implied by Theorem \ref{cycthm} if the system is reachable from and controllable to a ground state) this leads to the following cyclo-dissipativity consequence of the Second Law:

\smallskip
{\it Consider a thermodynamic system with a thermal port $(T,q)$ and a mechanical port $(-P,u_V)$. Then for every constant temperature $T$ the thermodynamic system is cyclo-passive with respect to the passivity supply rate $-Pu_V$; that is, for every $T$ there exists a function $F_T$ of the state $x$ of the thermodynamic system satisfying}
\bq
\label{helmholtz}
F_T(x(t_2)) - F_T(x(t_1))  \leq  \int_{t_1}^{t_2}-P(t)u_V(t) dt.
\eq
In fact, the function $F(x,T):=F_T(x)$ is known in thermodynamics as the (Helmholtz) {\it free energy}. For example, for a gas, given the energy $E(V,\cS)$ expressed as a function of the volume $V$ and the entropy $\cS$, the Helmholtz function is given as 
\bq
F(V,T) = E(V,\cS) - T\cS, \quad T=\frac{\partial E}{\partial \cS}(V,\cS),
\eq
where $\cS$ is solved from the equation $T=\frac{\partial E}{\partial \cS}(V,\cS)$. Thus $F(V,T)$ is the partial Legendre transform of $E(V,\cS)$ with respect to the entropy $\cS$.

From a general system-theoretic point of view this leads to the following question. Consider a {\it general} two-port cyclo-passive system as in Figure \ref{fig:two-port}. Under which conditions is it {\it not} possible to transform, in a recurrent manner, energy at port $1$ into energy at port $2$ while keeping $y_1$ constant? Or said differently, what is so special about thermodynamic systems and their limitations for converting heat into mechanical work, and are there any other systems than thermodynamic systems that can{\it not} transform energy from one port into the other while keeping the output at the first port constant? This question will be addressed in the next section.

\section{ONE-PORT CYCLO-PASSIVITY}\label{sec:one-port_pass}

Consider a general cyclo-passive physical system. It is well-known (see, e.g., \cite{jeltsema} and \cite{passivitybook}) that any such physical system can be naturally modeled as a {\it port-Hamiltonian system}. Throughout the rest of this paper we will concentrate on the class of input-state-output port-Hamiltonian systems without feedthrough \cite{passivitybook} given as
\bq
\label{pH}
\begin{aligned}
\dot{x} & =  J(x)e - \mathcal{R}(x,e) + G(x)u, \quad e= \frac{\partial H}{\partial x}(x), \\[2mm]
y & = G^T(x)\frac{\partial H}{\partial x}(x), \quad x \in \X,
\end{aligned}
\eq
with $n$-dimensional state space $\X$, Hamiltonian $H: \X \to \mR$, skew-symmetric matrix $J(x)$, and energy-dissipation mapping $\mathcal{R}$ satisfying
\bq
e^T \mathcal{R} (x,e) \geq 0, \mbox{ for all } x,e.
\eq
Furthermore, $u \in \mR^m$ denotes the vector of inputs and $y \in \mR^m$ the vector of outputs, together defining the port $(u,y)$. Obviously, port-Hamiltonian systems \eqref{pH} satisfy
\bq
\frac{d}{dt} H(x)= e^TJ(x)e - e^T \mathcal{R} (x,e) + e^TG(x)u \leq y^Tu,
\eq
with $y^Tu$ denoting the externally supplied power to the system. Hence port-Hamiltonian systems are {\it cyclo-passive} with (possibly indefinite) storage function $H$. 

In order to study limits to energy conversion we will further assume that the port-Hamiltonian system \eqref{pH} has {\it two} ports $(u_1,y_1)$ and $(u_2,y_2)$ as in Figure \ref{fig:two-port}. Consequently
\bq
\frac{d}{dt} H(x) \leq y_1^Tu_1 + y_2^Tu_2
\eq
Which conditions ensure that, while keeping $y_1$ constant, no energy can be transported from port $1$ to port $2$ in a recurrent manner? 

Before we formulate our main theorem let us recall the partial Legendre transformation. Consider a real-valued function $H(x_1,x_2)$ of the vectors $x_1$ and $x_2$, where it is assumed that the Hessian $\frac{\partial^2 H}{\partial x_1^2}(x_1,x_2)$ has full rank everywhere. Then the {\it partial Legendre transform}\footnote{Note that the Legendre transformation \eqref{legendre} adheres to the same sign convention as for defining the thermodynamic potentials, such as the Helmholtz function, in thermodynamics. This in contrast with the opposite sign convention in mechanics (from Lagrangian to Hamiltonian function) and convex analysis.} of $H$ with respect to $x_1$ is defined as
\begin{equation}
\label{legendre}
H^{*}_1(e_1,x_2) = H(x_1,x_2) - e_1^Tx_1, \quad e_1= \frac{\partial H}{\partial x_1}(x_1,x_2),
\end{equation}
where $x_1$ is expressed as a function of $e_1,x_2$ by means of the equation $e_1= \frac{\partial H}{\partial x_1}(x_1,x_2)$ (which is locally guaranteed by the full rank assumption on the Hessian matrix). The following properties of the partial Legendre transformation are well-known:
\bq
\label{legendreprop}
\frac{\partial H^{*}_1}{\partial e_1}(e_1,x_2)= - x_1, \quad
\frac{\partial H^{*}_1}{\partial x_2}(e_1,x_2)= \frac{\partial H}{\partial x_2}(x_1,x_2). 
\eq


\begin{theorem}\label{theorem}
Consider a port-Hamiltonian system \eqref{pH} with two ports $(u_1,y_1)$ and $(u_2,y_2)$. Suppose the state vector $x$ can be split as 
\begin{equation*}
x=\begin{bmatrix} x_1 \\ x_2 \end{bmatrix} 
\end{equation*}
in such a way that the port-Hamiltonian system takes the form
\bq
\label{form}
\begin{aligned}
\dot{x}_1 & = J_1(x_1,x_2)e_1 - \mathcal{R}_1(x_1,x_2,e_1) + G_1u_1,\\[1em]
\dot{x}_2 & = J_2(x_1,x_2)e_2 - \mathcal{R}_2(x_1,x_2,e_2) + G_2(x_1,x_2)u_2,\\[0.5em]
y_1 & = G_1^Te_1, \qquad \quad e_1= \frac{\partial H}{\partial x_1}(x_1,x_2), \\
y_2 & = G_2^T(x_1,x_2)e_2, \quad  e_2= \frac{\partial H}{\partial x_2}(x_1,x_2),
\end{aligned}
\eq
where $G_1$ is an invertible square constant matrix. Then, whenever $y_1=\bar{y}_1$ with $\bar{y}_1$ constant, also $e_1$ is equal to a constant $e_1=\bar{e}_1$, and furthermore
\bq
\frac{d}{dt} H^{*}_1(\bar{e}_1,x_2) \leq y^T_2u_2.
\eq
This implies that the system for any constant $y_1=\bar{y}_1$ is {\it cyclo-passive} at the second port $(u_2,y_2)$, with respect to the storage function $H^{*}_1(\bar{e}_1,x_2)$ (regarded as function of $x_2$).
\end{theorem}

\begin{proof}
By \eqref{legendreprop} 
\bq
\frac{d}{dt} H^{*}_1 =  - x^T_1 \dot{e}_1 + \frac{\partial H}{\partial x^T_2}(x_1,x_2)\dot{x}_2.
\eq
Hence, for constant $y_1=\bar{y}_1$ (and thus constant $e_1$) we obtain
\begin{align}
\label{L}
\frac{d}{dt} H^{*}_1 & = e^T_2 \big[J_2(x_1,x_2)e_2 - \mathcal{R}_2(x_1,x_2,e_2) + G_2(x_1,x_2)u_2 \big] \nonumber\\[2mm]
&= - e_2^T\mathcal{R}_2(x_1,x_2,e_2) + y^T_2u_2 \leq y^T_2u_2.
\end{align}
\end{proof}

The property of being cyclo-passive at the second port $(u_2,y_2)$ for any constant value $\bar{y}_1$ of the output $y_1$ of the first port will be called {\it one-port cyclo-passivity} at the second port.

\begin{definition}
Consider a cyclo-passive system with two ports $(u_1,y_1)$ and $(u_2,y_2)$, i.e., there exists a function $S: \X \to \mathbb{R}$ such that along every solution
\bq 
\frac{d}{dt}S \leq y_1^Tu_1 + y_2^Tu_2.
\eq
Then, the system is said to be {\it one-port cyclo-passive} at the port $(u_2,y_2)$ if for every constant $\bar{y}_1$ there exists an $S_1(x)$ such that for all trajectories for which $y_1(t)=\bar{y}_1$
\bq 
\frac{d}{dt} S_1 \leq y_2^Tu_2.
\eq
\end{definition}
Hence, Theorem \ref{theorem} states that any port-Hamiltonian system in the form \eqref{form} is one-port cyclo-passive at the second port $(u_2,y_2)$, with storage function given by the partial Legendre transform $H_1^*$ of $H$ with respect to $x_1$.
\begin{remark}\rm
Restricting to the partial Legendre transform $H_1^*$ as candidate storage function it can be seen that the sufficient conditions of Theorem \ref{theorem} are close to necessary as well. For example, if there are off-diagonal blocks $J_{12}=-J^T_{21}$ in the $J$-matrix then \eqref{L} extends to
\bq
\begin{aligned}
\frac{d}{dt} H^{*}_1  & = e^T_2 \big[J_2(x_1,x_2)e_2  + J_{21}(x_1,x_2)e_1- \mathcal{R}_2(x_1,x_2,e_2)\big] \\[2mm]
& \quad + \, G_2(x_1,x_2)u_2,
\end{aligned}
\eq
for constant $e_1$, which is $\leq y^T_2u_2$ for all $e_1,e_2$ if and only if $J_{12}=-J^T_{21}=0$.
\end{remark}
\begin{remark}\rm
\label{remark}
A somewhat complementary (and easier) statement, which does not require the Legendre transformation of $H$, is that by selecting $u_1$ such that $\dot{x}_1$ in \eqref{form} is zero, and thus $x_1=\bar{x}_1$ is constant, then the system at the $(u_2,y_2)$-port is cyclo-passive with storage function $H(\bar{x}_1,x_2)$. Indeed, in this case 
\bq
\frac{d}{dt}H(\bar{x}_1,x_2) \leq y^T_2u_2.
\eq
A similar situation will be encountered in Section \ref{subsec:HEX}.
\end{remark}

\section{EXAMPLES}\label{sec:examples}

In this section, a variety of examples from different physical domains are considered from the perspective of Theorem \ref{theorem}.

\subsection{Ideal gas}
\label{gas}
The dynamics of a gas are described by
\bq
\begin{aligned}
\dot{V} & = u_V,\\[1em]
\dot{S} & = u_S,\\[0.5em] 
y_V &= -P = \frac{\partial E}{\partial V}(V,S),\\[0.5em]
y_S &= T = \frac{\partial E}{\partial S}(V,S),
\end{aligned}
\eq
where $E(V,S)$ is the energy of the gas, $u_V$ is the rate of extension of the volume, and $u_S$ is the entropy flow. It is assumed that the entropy flow is given by $u_S = \frac{q}{T}$, i.e., only {\it reversible} processes are considered. In case of an ideal gas \cite{fermi, kondepudi}
\bq
E(V,S)= \frac{C_Ve^{\frac{S}{C_V}}}{Ve^{\frac{R}{C_V}}},
\eq
where $C_V$ denotes the heat capacity (at constant volume), and $R$ is the universal gas constant. 

Clearly, this thermodynamic system satisfies the conditions of Theorem \ref{theorem}, and thus the system is one-port cyclo-passive (in fact, {\it cyclo-lossless}) at the {\it mechanical} port for any constant temperature $T=\bar{T}$, with storage function given by the Helmholtz function $F(V,\bar{T})$, which for any $T$ is given as \cite{fermi}
\bq
F(V,T) = C_VT + W - T\big(C_V \ln T + R \ln V +a\big),
\eq
with $a$ the entropy constant of the gas, and $W$ a constant of integration. 

Note that the system also satisfies the conditions of Theorem \ref{theorem} with the two ports reversed. Thus, the system is also one-port lossless at the {\it thermal} port for any constant pressure $P$. 

From the point of view of the well-known {\it Carnot cycle} \cite{fermi, kondepudi} we notice that dynamics for $T$ constant (the case covered by Theorem \ref{theorem}) corresponds to an {\it isothermal} process, while the dynamics for $S$ constant (the case covered by Remark \ref{remark}) corresponds to an {\it isentropic} (or, adiabatic) process.
%

\subsection{DC-motor}
The standard model of a DC-motor as depicted in Fig.~\ref{fig:dcmotor} is given in port-Hamiltonian formulation as \cite{jeltsema}
\bq
\label{eq:pHDCmotor}
\begin{aligned}
\begin{bmatrix}
\dot{\varphi}\\
\dot p
\end{bmatrix}
&=
\left(
\begin{bmatrix}
0 & -K\\
K & 0
\end{bmatrix}
-
\begin{bmatrix}
R & 0\\
0 & b
\end{bmatrix}
\right)
\begin{bmatrix}
\dfrac{\varphi}{L}\\[1em] 
\dfrac{p}{J} 
\end{bmatrix}
+
\begin{bmatrix}
1 & 0 \\ 0 & 1
\end{bmatrix}
\begin{bmatrix} V \\ \tau \end{bmatrix}, \\
\begin{bmatrix} I \\ \omega \end{bmatrix} &=
\begin{bmatrix}
  1 & 0 \\0 & 1
\end{bmatrix}
\begin{bmatrix}
\dfrac{\varphi}{L}\\[1em] 
\dfrac{p}{J} 
\end{bmatrix}, \quad H(\varphi,p) = \frac{\varphi^2}{2L} + \frac{p^2}{2J},
\end{aligned}
\eq
where the state variables are the flux linkage $\varphi$ and the angular momentum $p$. The inductance is $L$, the moment of inertia $J$, while $K \neq 0$ is the gyration constant, which is responsible for the conversion of electrical power into mechanical power and conversely. The electrical port is described by the pair $(V,I)$, i.e., the voltage and the current, while the mechanical port is $(\tau, \omega)$, i.e., the torque and the angular velocity. 

\begin{figure}[h]
\centering
\psfrag{U}[][]{$V$}
\psfrag{i}[][]{$I$}
\psfrag{J}[][]{$J$}
\psfrag{b}[][]{$b$}
\psfrag{R}[][]{$R$}
\psfrag{L}[][]{$L$}
\psfrag{K}[][]{$K$}
\psfrag{w}[][]{$\omega$}
\psfrag{T}[][]{$\tau$}
\psfrag{+}[][]{$+$}
\psfrag{-}[][]{$-$}
\includegraphics[width=6cm]{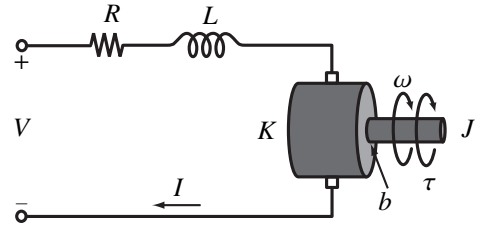}
\caption{DC motor.}
\label{fig:dcmotor}
\end{figure}
Note that the same model can be used for the operation of the DC-motor as a {\it dynamo}; converting instead mechanical energy into electrical energy.

The partial Legendre transforms of the total energy $H$ with respect to $\varphi$, respectively $p$, are given as 
\bq
H_1^*(I,p) =  \frac{p^2}{2J} - \frac{1}{2}LI^2, \quad H_2^*(\varphi,\omega) =  \frac{\varphi^2}{2L} - \frac{1}{2}J\omega^2.
\eq
In both cases the system does {\it not} satisfy the conditions of Theorem \ref{theorem}. In fact, the system is {\it not} one-port cyclo-passive, either at the mechanical port (if $\bar{I} \neq 0$ ) or at the electrical port (if $\bar{\omega} \neq 0$). Indeed, consider the system with $I=\bar{I} \neq 0$. Then a storage function $S(p)$ for this constrained system should satisfy
\bq
\frac{d}{dt} S = \frac{dS}{dp} \left[K \frac{\bar{\varphi}}{L} - b \frac{p}{J} + \tau \right] \leq \omega \tau
\eq
for all $\tau$, where $\frac{\bar{\varphi}}{L}=\bar{I}$. It follows that $\frac{dS}{dp}=\omega$, and thus that 
\begin{equation*}
S(p)= \frac{p^2}{2J} + \text{const}. 
\end{equation*}
After substitution this implies
\bq
\frac{dS}{dp} \left[K \bar{I} - b\frac{p}{J}\right]= \omega K\bar{I} -b\omega^2 \leq 0
\eq
for all $\omega$. However, this can only be true whenever either $\bar{I}=0$ or $K=0$. Thus the system is {\it not} one-port cyclo-passive at the mechanical port for $I=\bar{I} \neq 0$. 

A similar argument can be used for the case $\omega=\bar{\omega} \neq 0$, showing that the system is not one-port cyclo-passive at the electrical port as well. Thus the DC-motor can be used for converting electrical energy into mechanical energy, as well as (in dynamo mode) for converting mechanical into electrical energy; in agreement with common usage. 

\subsection{Adjustable spring}

Consider a spring with extension $q$ satisfying Hooke's law, where the spring stiffness $k$ is adjustable as a function of time; see Fig.~\ref{fig:varspring}. 

\begin{figure}[h]
\begin{center}
\psfrag{F}[][]{$F$}
\psfrag{K}[][]{$k(t)$}
\psfrag{q}[][]{$q$}
\includegraphics[width=0.5\columnwidth]{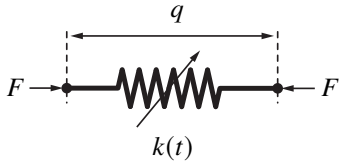}
\end{center}
\caption{A linear adjustable spring.}
\label{fig:varspring}
\end{figure}

\noindent This defines the port-Hamiltonian system
\bq
\begin{aligned}
\dot{q} & = v,\\[1em]
\dot{k} & = u,\\[0.66em]
F &= \frac{\partial H}{\partial q}(q,k),\\[0.5em]
y &= \frac{\partial H}{\partial k}(q,k),
\end{aligned}
\eq
with Hamiltonian $H(q,k)=\frac{1}{2}kq^2$, inputs $v,u$, and outputs $F,y$. In robotics this is used in 'variable stiffness control'; see e.g. Note 6.13 in \cite{passivitybook}.

Although the system is in the form \eqref{form}, it does {\it not} satisfy the conditions of Theorem \ref{theorem} for cyclo-passivity at the $(v,F)$-port, since the partial Legendre transformation of $H(q,k)$ with respect to $k$ is singular. Indeed, since $H(q,k)$ is homogeneous of degree one in $k$, the Hessian 
\begin{equation*}
\frac{\partial^2 H}{\partial k^2}(q,k) = 0, 
\end{equation*}
and thus the partial Legendre transform with respect to $k$ is zero.

On the other hand, the partial Legendre transformation with respect to $x$ is given as $H^*(F,k)=- \frac{1}{2k}F^2$, and the system for constant $F$ is one-port cyclo-passive (in fact, cyclo-lossless) at the $(u,y)$-port. 

\subsection{Coupled inductors / transformer}

A pair of magnetically coupled inductors can be considered as a non-ideal AC transformer \cite{Desoer}; see Fig.~\ref{fig:transformer}.

\begin{figure}[h]
\begin{center}
\psfrag{a}[][]{$V_1$}
\psfrag{b}[][]{$I_1$}
\psfrag{c}[][]{$I_2$}
\psfrag{d}[][]{$V_2$}
\psfrag{L}[][]{$L_1$}
\psfrag{K}[][]{$L_2$}
\psfrag{m}[][]{$M$}
\includegraphics[width=0.45\columnwidth]{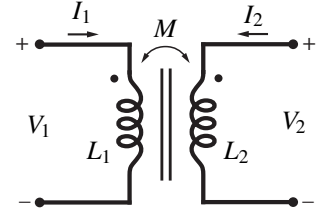}
\end{center}
\caption{Two magnetically coupled inductors.}
\label{fig:transformer}
\end{figure}

Let $\psi_k$, with $k=1,2$, represent the associated flux-linkages, the port-Hamiltonian dynamics is given by
\bq \label{eq:trans_states}
\begin{aligned}
\dot{\psi}_k &= V_k, \quad k=1,2,
\end{aligned}
\eq
together with the outputs
\bq
\begin{aligned}
I_k &= \frac{\partial H}{\partial \psi_k},  \quad k=1,2,
\end{aligned}
\eq
and as Hamiltonian the total stored magnetic energy 
\bq
H(\psi_1,\psi_2) = \frac{1}{2}[\psi_1 \ \psi_2] \begin{bmatrix} L_1 & M\\M & L_2 \end{bmatrix}^{-1} \begin{bmatrix} \psi_1 \\ \psi_2 \end{bmatrix}.
\eq
Under the condition that ${|M| \in \left[0,\sqrt{L_1L_2}\,\right)}$ (imperfect coupling), the Hessian of $H$ has full rank. Hence, the partial Legendre transform of $H$ with respect to $\psi_1$ reads
\begin{equation*}
\begin{aligned}
H_1^*(I_1,\psi_2) &= H(\psi_1,\psi_2) - I_1 \psi_1\Big|_{\psi_1 = L_1 I_1 + \frac{M}{L_2}(\psi_2 - MI_1)}\\
&= \frac{M^2 - L_1L_2}{2L_2}I_1^2 - \frac{M}{L_2}I_1\psi_1 + \frac{1}{2L_2}\psi_2^2,
\end{aligned}
\end{equation*}
which, using \eqref{eq:trans_states}, readily implies that for any $I_1 = \bar{I}_1$
\bq
\frac{d}{dt}H_1^*(\bar{I}_1,\psi_2) = V_2I_2.
\eq
I.e., a non-ideal AC transformer with a constant primary current is one-port cyclo-passive (in fact, lossless) with respect to the secondary port $(V_2,I_2)$. The same holds for the primary port $(V_1,I_1)$ when the secondary current is held constant, i.e., for any $I_2 = \bar{I}_2$. Physically this is obvious since a transformer cannot transfer a DC current.  

\subsection{Capacitor microphone}

Consider a capacitor microphone depicted in Fig.~\ref{fig:CapMic}. The capacitance $C(q) > 0$ is varying as a function of the
displacement $q$ of the moving plate with mass $m$, which is attached to a
spring ${k > 0}$ and a damper ${b>0}$, and affected by a mechanical force $F$, i.e., air pressure arising from sound. The capacitance increases when the distance between the plates becomes smaller. Furthermore, $E$ is an externally supplied voltage\footnote{In the audio industry this voltage is usually $48$ volts (DC) and is called phantom power.} and the resistor ${R>0}$ is used to convert the current into a voltage that is sent to an amplifier.    

\begin{figure}[h]
\begin{center}
\psfrag{E}[][]{$E$}
\psfrag{V}[][]{$V$}
\psfrag{F}[][]{$F$}
\psfrag{x}[][]{$q$}
\psfrag{q}[][]{$\dot{Q}$}
\psfrag{K}[][]{$k$}
\psfrag{B}[][]{$b$}
\psfrag{C}[][]{$C$}
\psfrag{R}[][]{$R$}
\psfrag{M}[][]{$m$}
\includegraphics[width=0.85\columnwidth]{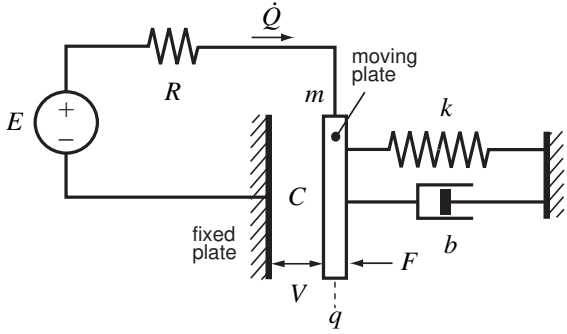}
\end{center}
\caption{Capacitor microphone.}
\label{fig:CapMic}
\end{figure}

Let $p$ be the momentum of the moving plate and $Q$ the electrical charge, then the equations of motion can be written as the port-Hamiltonian system \cite{passivitybook}
\begin{equation}
\label{eq4.2.65}
\begin{aligned}
\left[\begin{array}{rcl}
\dot{q} \\
\dot{p} \\
\dot{Q}
\end{array}
\right] & =
\left( \left[ 
\begin{array}{ccc}
0 & 1 & 0 \\ 
-1 & 0 & 0 \\ 
0 & 0 & 0
\end{array}
\right] -\left[ 
\begin{array}{ccc}
0 & 0 & 0 \\ 
0 & b & 0 \\ 
0 & 0 & 1/R
\end{array}
\right] \right)
\left[ 
\begin{array}{c}
\dfrac{\partial H}{\partial q} \\[3mm]
\dfrac{\partial H}{\partial p} \\[3mm]
\dfrac{\partial H}{\partial Q}
\end{array}
\right]\\ 
& \qquad +\left[ 
\begin{array}{c}
0 \\ 
1 \\ 
0
\end{array}\right] 
F+\left[ 
\begin{array}{c}
0 \\ 
0 \\ 
1\\
\end{array}
\right] I, \\[6mm] 
v &=\frac{\partial H}{\partial p}, \\[2mm]
V & = \frac{\partial H}{\partial Q},
\end{aligned}
\end{equation}
with $v=\dot{q}$ the velocity of the moving plate, where we have applied a Thevenin-Norton transformation \cite{Desoer} to convert the voltage source $E$ in series with the resistor $R$ into a parallel connection of a current source $I=E/R$ and a resistor with the same resistance. Furthermore, the Hamiltonian $H$ represents the total energy 
\begin{equation}
\label{eq4.2.66}
H(q,p,Q)=\frac{1}{2}kq^{2} + \frac{p^{2}}{2m}+\frac{Q^{2}}{2C(q)}
\end{equation} 
Hence, the power-balance is
\bq
\frac{d}{dt}H = - b\dot{q}^2 - RI^2 + Fv + IE  \leq Fv + IE,
\eq
with $Fv$ the mechanical power and $IE$ the electrical
power supplied to the system.\footnote{Note that the same model can be used for electro-mechanical micro-actuators (e.g., MEMS devices) that are controlled at the electrical port in order to produce a certain desired force at the mechanical port; see \cite{MEMS_book} for more details.} 

It follows from Theorem \ref{theorem} that the system for a constant capacitor voltage ${V=\bar{V}}$ is one-port cyclo-passive at the mechanical port $(F,v)$, with storage function given by the partial Legendre transform of $H$ with respect to $Q$, i.e.,
\bq
H^*_{Q}(q,p,\bar{V}) := \frac{1}{2}kq^2 + \frac{p^2}{2m} - \frac{1}{2}C(q)\bar{V}^2.
\eq
Furthermore, under the additional assumption that there exists a constant $\kappa$ such that $0\leq C(q) \leq \kappa$ for all $q$, then $H^*_{Q}$, as a function of $(q,p)$, is bounded from below, and thus the system is actually {\it one-port passive} at the mechanical port.

On the other hand, the system does {\it not} satisfy the sufficient conditions for one-port cyclo-passivity at the electrical port. In fact, it can be proven that the system is {\it not} one-port cyclo-passive at the electrical port for any non-zero constant mechanical velocity $\bar{v}$, provided that $C'(q) \neq 0$ (i.e., the capacitance non-trivially depends on $q$).  

Indeed, suppose there exists $K(q,Q)$ such that for $v = \bar{v} = \bar{p}/m \neq 0$, we have that
\bq\label{eq:dK/dt}
\frac{d}{dt}K = \frac{\partial K}{\partial q}\dot{q} + \frac{\partial K}{\partial Q}\dot{Q} \leq IV.
\eq
Since $\dot{q}=\bar{v}$ this implies
\bq
\frac{\partial K}{\partial q}\bar{v} + \frac{\partial K}{\partial Q}\left[- \frac{V}{R} + I \right] \leq IV
\eq
for all currents $I$, and thus
\bq\label{eq:dK_equality}
\frac{\partial K}{\partial Q} =V
\eq
and
\bq\label{eq:dK_inequality}
\frac{\partial K}{\partial q}\bar{v} - \frac{\partial K}{\partial Q}\frac{V}{R} \leq 0.
\eq
Now, the equality (\ref{eq:dK_equality}) implies $\frac{\partial K}{\partial Q} = \frac{\partial H}{\partial Q}$, and thus that
\begin{equation*}
K(q,Q)= \frac{Q^2}{2C(q)} +G(q)
\end{equation*}
for some function $G(q)$. Substitution into (\ref{eq:dK_inequality}) yields the differential inequality
\bq\label{eq:diff_inequality1}
\frac{d}{dq}\left[G(q) - C(q)\frac{V^2}{2} \right]\bar{v} \leq \frac{V^2}{R}
\eq
for all $q,V$. In the trivial case when $C(q)$ is constant, i.e., {\it not} depending on $q$, we can simply select ${G(q)=0}$ (or any arbitrary constant), in which case the required inequality is obviously true. However, for the case that $C(q)$ \emph{is} depending on $q$, the inequality (\ref{eq:diff_inequality1}) does not allow for a solution $G(q)$ for any $\bar{v} \neq 0$. Thus, there does \emph{not} exist a suitable storage function $K(q,Q)$ that satisfies (\ref{eq:dK/dt}) for arbitrary velocity $\bar{v} \neq 0$ and non-constant capacitance $C(q)>0$. 

Physically this is due to the fact that if $\bar{v} \neq 0$ the capacitor plate is moving at a constant speed. A moving plate under a constant excitation by the voltage source means that the charge is constantly increasing (or decreasing, depending on the sign of $\bar{v}$). At the same time, also the potential energy stored by the spring is changing but this is not caused by the electrical port.   Hence, from the perspective of the electrical port, when $\bar{v}> 0$ the system is \emph{creating} internal energy due to the increasing capacitance and the increase of the potential energy of the spring.

\subsection{Synchronous machine}
The standard model of a synchronous machine, see e.g. \cite{kundur}, can be naturally written into port-Hamiltonian form (see \cite{EJCFiaz} for details)
\bq
\label{sg}
\begin{aligned}
\begin{bmatrix}
\dot{\psi}_s \\[2mm]
\dot{\psi}_r \\[2mm]
\dot{p} \\[2mm]
\dot{\theta} 
\end{bmatrix} & =  
\begin{bmatrix}
- R_s & 0_{33} & 0_{31} &0_{31} \\[2mm]
0_{33} & -R_r & 0_{31} &0_{31} \\[2mm]
0_{13} & 0_{13} & -b & -1 \\[2mm]
0_{13} & 0_{13} & 1 & 0 \end{bmatrix}
\begin{bmatrix}
\dfrac{\partial H}{\partial \psi_s} \\[3mm]
\dfrac{\partial H}{\partial \psi_r} \\[3mm]
\dfrac{\partial H}{\partial p} \\[3mm]
\dfrac{\partial H}{\partial \theta} \end{bmatrix}\\ 
& \qquad +
\begin{bmatrix}
I_{3} & 0_{31} & 0_{31} \\[2mm] 
0_{33} & e_1 & 0_{31} \\[2mm] 
0_{13} & 0 & 1 \\[2mm] 
0_{13} & 0 & 0
\end{bmatrix} \! \begin{bmatrix} V_s \\[2mm] V_f \\[2mm] \tau \end{bmatrix}, 
\end{aligned}
\eq
together with the outputs
\bq
\label{sg_outputs}
\begin{aligned}
\begin{bmatrix}
I_s \\[2mm] I_f \\[2mm] \omega \end{bmatrix} & = 
\begin{bmatrix}
I_{3} & 0_{33} & 0_{31} & 0_{31} \\[2mm]
0_{13} & e_1^T & 0 & 0  \\[2mm]
0_{13} & 0_{13} & 1 & 0 \end{bmatrix}
\begin{bmatrix}
\dfrac{\partial H}{\partial \psi_s} \\[3mm]
\dfrac{\partial H}{\partial \psi_r} \\[3mm]
\dfrac{\partial H}{\partial p} \\[3mm]
\dfrac{\partial H}{\partial \theta} \end{bmatrix},
\end{aligned}
\eq
where $0_{lk}$ denotes the $l \times k$ zero matrix, $I_3$ denotes the $3 \times 3$ identity matrix, and $e_1$ is the first basis vector of $\mR^3$. Furthermore, $R_s$ and $R_r$ are positive diagonal $3 \times 3$ matrices, modeling the internal energy dissipation. The state variables of this $8$-dimensional model comprise
\begin{itemize}
\item
$\psi_s \in \mathbb{R}^3$, the stator fluxes;
\item
$\psi_r \in \mathbb{R}^3$, the rotor fluxes;
the first one corresponding to the field winding and the remaining two to the damper windings,
\item
$p \in \mathbb{R}$, the angular momentum of the rotor;
\item
$\theta \in \mathbb{R}$, the angle of the rotor.
\end{itemize}
Furthermore, $V_s, I_s \in \mathbb{R}^3$ are the three-phase
stator terminal voltages and currents, $V_f, I_f \in \mathbb{R}$ are the rotor field winding voltage and current, and $\tau, \omega \in \mathbb{R}$ are the mechanical torque and angular velocity; see Fig.~\ref{fig:1} for a schematic view. The synchronous machine is (like the DC-motor as treated before) a clear example of an energy-converting device. If the aim is to convert mechanical energy into electrical energy then it is called a synchronous {\it generator}, while if the aim is to convert electrical into mechanical energy it is called a synchronous {\it motor}. Fig.~\ref{fig:1} shows the operation as a synchronous generator.

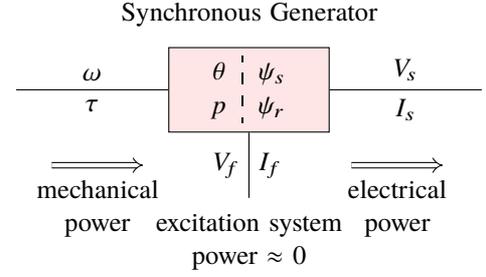
\begin{figure}[t]
\begin{center}
\begin{tikzpicture}[auto, node distance=2cm,>=latex', scale = 1.0, transform shape]

\tikzstyle{block} = [draw, fill=red!10, rectangle, 
    minimum height=3em, minimum width=6em]
\tikzstyle{position} = [coordinate]
\tikzstyle{pinstyle} = [pin edge={to-,thin,black}]

    \node (SG) {Synchronous Generator};
    \node [block, below of=SG, 
                node distance=1cm] (system) {\begin{tabular}{c:c}
                $\theta$ & $\psi_s$ \\ 
                $p$ & $\psi_r$ \\ 
                \end{tabular}
                };
    \node[position, left =2cm of system](left1){};
    \node[position, right =2cm of system](right1){};  
    \node[below of = system] (power) {\begin{tabular}{cc}  excitation system \\ power $\approx$ 0\end{tabular}};
    
    \draw [-] (left1) -- node[above] {$\omega$} node[below]{$\tau$}++ (system);    
    \draw [-] (right1) -- node[above] {$V_s$} node[below]{$I_s$} ++ (system);   
    \draw [-] (power) -- node[left, name=vf] {$V_f$} node[right, name=if]{$I_f$} ++ (system);

    \node[position, left = 2cm of vf](left2){};
    \node[position, left = 0.8cm of vf](left3){};
    \node[position, right = 0.8cm of if](right2){};
    \node[position, right = 2cm of if](right3){};
    
    \draw[-implies,double equal sign distance] (left2) -- node[below]{\begin{tabular}{cc}
        mechanical \\ power
        \end{tabular}
        } (left3); 
        \draw[-implies,double equal sign distance] (right2) -- node[below]{\begin{tabular}{cc}
            electrical \\ power
        \end{tabular}
        } (right3); 
\end{tikzpicture}
\end{center}
\caption{The state and port variables of the synchronous generator}
\label{fig:1}
 \end{figure}

The Hamiltonian $H$ of the synchronous machine is the sum of the magnetic energy of the field between stator and rotor, and the kinetic energy of the rotor, i.e.,
\bq
\label{eq5}
H(\psi_s, \psi_r, p, \theta) =  \frac{1}{2} 
\begin{bmatrix} \psi_s^T & \psi_r^T \end{bmatrix} L^{-1}(\theta) \begin{bmatrix} \psi_s \\ \psi_r \end{bmatrix} + 
 \frac{1}{2J_r}p^2,
 \eq
where $J_r>0$ is the rotational inertia of the rotor 
and $L(\theta) \succ 0$ is a $6 \times 6$ positive-definite symmetric inductance matrix; see \cite{EJCFiaz} for details.  

It is directly seen that the synchronous generator satisfies the conditions of Theorem \ref{theorem} for one-port cyclo-passivity at the mechanical port, with storage function given by the partial Legendre transform $H^*_{\psi_s}(I_s,\psi_r,p, \theta)$ of the Hamiltonian $H(\psi_s, \psi_r, p, \theta)$ with respect to the stator fluxes $\psi_s$. Partitioning the $6 \times 6$ inductance matrix $L(\theta)$ corresponding to $\psi_s, \psi_r$ as (while suppressing in the notation the dependence on $\theta$)
\[
L= \begin{bmatrix} L_{ss} & L_{sr} \\ L_{rs} & L_{rr}\end{bmatrix},
\]
this partial Legendre transform can be computed as
\bq
\begin{aligned}
H^*_{\psi_s}&(I_s,\psi_r,p, \theta) = 
\frac{1}{2J_r}p^2 +  \frac{1}{2}\psi_r^TL_{rr}^{-1}\psi_r \\
& - \frac{1}{2}I_s^T\left(L_{ss} - L_{sr}L_{rr}^{-1}L_{rs}\right)I_s + \psi_r^TL_{rr}^{-1}L_{rs}I_s. 
\end{aligned}
\eq
Here it can be noted that the Schur complement $L_{ss} - L_{sr}L_{rr}^{-1}L_{rs}$ is a positive-definite $3\times 3$ matrix (since the original inductance matrix $L$ is positive-definite), while also $L_{rr}^{-1}$ is positive-definite.

The physical interpretation of one-port cyclo-passivity at the mechanical port is obvious: if the stator currents $I_s$ are constant in time then no air-gap torque does arise, and hence there is no energy transfer from the electrical to the mechanical port. Furthermore, since the dependence of the elements of the inductance matrix $L(\theta)$ on the rotor angle $\theta$ is through products of $\cos$ and $\sin$ functions, these elements are bounded, implying that for constant $I_s$ the storage function $H^*_{\psi_s}(I_s,\psi_r,p, \theta)$ as a function of $\psi_r$, $p$ and $\theta$ is bounded from below. This means that actually the synchronous machine is one-port {\it passive} at the mechanical port for any constant $I_s$.

On the other hand, the synchronous machine does {\it not} satisfy the conditions of Theorem \ref{theorem} for one-port cyclo-passivity at the electrical port. In fact, lack of one-port cyclo-passivity at the electrical port is physically\footnote{Showing directly from the definition that the synchronous generator is not one-port cyclo-passive at the electrical port is not so obvious!} clear since in normal operation the synchronous generator continuously converts energy at the mechanical port for constant torque $\tau$ into electrical energy at the stator (to be supplied to the electrical grid).

\subsection{Synchronous machine in $dq$-coordinates}
Consider the synchronous machine \eqref{sg} as before. Apply the {\it Blondel-Park transformation} \cite{kundur, Machowski}:

\bq
\label{eq14}
\psi_{dq0} = \begin{bmatrix} \psi_d \\ \psi_q \\ \psi_0 \end{bmatrix} := \mathbb{T}_{dq0}(\theta) \psi_s,
 \eq
where 
\begin{equation*}
\mathbb{T}_{dq0}(\theta) = \sqrt{\frac{2}{3}}
\begin{bmatrix}
\cos \theta & \cos (\theta - \frac{2 \pi}{3}) &  \cos (\theta +  \frac{2 \pi}{3})\\
\sin \theta  & \sin (\theta - \frac{2 \pi}{3}) & \sin (\theta + \frac{2 \pi}{3}) \\
\frac{1}{\sqrt{2}} & \frac{1}{\sqrt{2}} & \frac{1}{\sqrt{2}} 
\end{bmatrix}.
\end{equation*}
Likewise, the voltages $V_s$ transform to $V_{dq0} := \mathbb{T}_{dq0}(\theta)V_s$, while the currents $I_s$ transform to $I_{dq0} := \mathbb{T}^{-T}_{dq0}(\theta)I_s =\mathbb{T}_{dq0}(\theta) I_s$. The Hamiltonian $H$ transforms to a Hamiltonian $\mathcal{H}$ in the new coordinates $(\psi_{dq0},\psi_r,p,\theta)$, whose magnetic part is given by
\bq
\label{eq16}
\mathcal{H}_m = \frac{1}{2} \begin{bmatrix} \psi_{dq0}^T & \psi_r^T \end{bmatrix} \mathcal{L}^{-1} \begin{bmatrix} \psi_{dq0} \\ \psi_r \end{bmatrix},
\eq
where $\mathcal{L}$ is a {\it constant} $6 \times 6$ matrix. 

It follows (see \cite{EJCFiaz, stegink} for details) that the model \eqref{sg} reduces to the $6$-dimensional port-Hamiltonian system
\bq
\label{sg_dq_state}
\begin{aligned}
\begin{bmatrix}
\dot{\psi}_{d}\\
\dot{\psi}_q \\
\dot{\psi}_r \\
\dot{p} 
\end{bmatrix} & = 
\begin{bmatrix}
- \begin{bmatrix}
r_s & 0 \\
0 & r_s
\end{bmatrix}
& 0_{23} & \begin{bmatrix} - \psi_q \\ \psi_d \end{bmatrix}  \\[4mm]
0_{32} & -R_r & 0_{31}  \\[3mm]
\begin{bmatrix} \psi_q & - \psi_d \end{bmatrix}& 0_{13} & -d  \end{bmatrix}
\begin{bmatrix}
\dfrac{\partial \widehat{\mathcal{H}}}{\partial \psi_d} \\[3mm]
\dfrac{\partial \widehat{\mathcal{H}}}{\partial \psi_q} \\[3mm]
\dfrac{\partial \widehat{\mathcal{H}}}{\partial \psi_r} \\[3mm]
\dfrac{\partial \widehat{\mathcal{H}}}{\partial p}  
\end{bmatrix}\\
& \qquad  
+ 
 \begin{bmatrix} I_2 & 0_{21} & 0_{21} \\ 0_{32} & e_1 & 0_{31} \\ 0_{12} & 0 & 1\end{bmatrix} 
 \begin{bmatrix} V_{dq} \\V_f \\ \tau \end{bmatrix}
\end{aligned}
\eq
together with the outputs
\bq
\label{sg_dq_output}
\begin{aligned}
\begin{bmatrix}
I_{dq} \\ I_f \\ \omega \end{bmatrix} & =
\begin{bmatrix}
I_{2} & 0_{23} & 0_{21} \\
0_{12} & e_1^T &0   \\
0_{12} &  0_{13} &1  \end{bmatrix}
\begin{bmatrix}
\dfrac{\partial \widehat{\mathcal{H}}}{\partial \psi_d} \\[3mm]
\dfrac{\partial \widehat{\mathcal{H}}}{\partial \psi_q} \\[3mm]
\dfrac{\partial \widehat{\mathcal{H}}}{\partial \psi_r} \\[3mm]
\dfrac{\partial \widehat{\mathcal{H}}}{\partial p}  \end{bmatrix} ,
\end{aligned}
\eq
where $V_{dq}, I_{dq} \in \mathbb{R}^2$ are the first two components of $V_{dq0}$, respectively $I_{dq0}$. Furthermore, the Hamiltonian $\widehat{\mathcal{H}}$ of the $6$-dimensional port-Hamiltonian system is equal in value to the Hamiltonian $\mathcal{H}$ of the original $8$-dimensional model minus a quadratic term in $\psi_0$.

Thus the Blondel-Park transformation eliminates the dependence of the magnetic energy on the rotor angle $\theta$ at the expense of the introduction of extra off-diagonal terms in the $J$-matrix in \eqref{sg_dq_state}. As a result the system does {\it not} satisfy anymore the conditions of Theorem \ref{theorem} for one-port cyclo-passivity at the mechanical port (contrary to the situation considered in the previous subsection). The physical reason is again obvious. Constant $I_{dq}$ means that there is an {\it alternating} current with constant {\it amplitude} at the stator side, which {\it does} produce a mechanical torque at the mechanical port. In fact, this is the standard operation of the synchronous machine as a synchronous {\it motor}; converting electrical energy into mechanical energy.

\subsection{Heat exchanger}\label{subsec:HEX}

As a final example, let us consider a single compartment heat exchanger in which thermal energy is exchanged between a \emph{cold} fluid stream, with enthalpy $H_c$ and mass flow rate $Q_c$, and a \emph{hot} fluid stream, with enthalpy $H_h$ and mass flow rate $Q_h$; see Fig.~\ref{fig:heatex}. This example is {\it not} an illustration of Theorem \ref{theorem}, but rather\footnote{Strictly speaking the system is {\it not} in the form \eqref{form} because of the off-diagonal terms in the $J$-matrix. Nevertheless, the same conclusion as in Remark \ref{remark} can be drawn due to the special form of the Hamiltonian.} of Remark \ref{remark}.

\begin{figure}[h]
\begin{center}
\psfrag{a}[r][]{$u_c = \dfrac{Q_c}{\rho_c V_c}$}
\psfrag{b}[][]{$y_c$}
\psfrag{c}[l][]{$u_h = \dfrac{Q_h}{\rho_h V_h}$}
\psfrag{d}[][]{$y_h$}
\includegraphics[width=0.6\columnwidth]{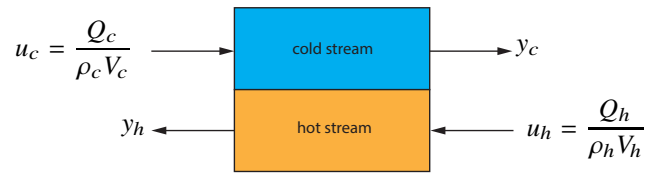}
\end{center}
\caption{Single compartment heat exchanger with volumes $V_c=V_h=:V$, densities $\rho_c$ and $\rho_h$, and mass flow rates $Q_c$ and $Q_h$. }
\label{fig:heatex}
\end{figure}

Let $T_c$ and $T_h$ represent the temperatures of the cold and hot streams, $H_c^\text{in}$ and $H_h^\text{in}$ the enthalpies at the respective inlets, and define the state of the heat exchanger as
\begin{equation}\label{eq:HEX_state}
x := \begin{bmatrix}
H_c\\
H_h
\end{bmatrix}.
\end{equation}
Following \cite{Zitte2019}, select the Hamiltonian $H = H_c + H_h$ and the total entropy $S = S_c + S_h$, with 
\begin{equation*} 
\frac{\partial S}{\partial x}(x) = 
\begin{bmatrix}
1/T_c\\
1/T_h
\end{bmatrix}.
\end{equation*}
Then, under standard ideality assumptions, the heat exchanger dynamics can be rewritten as the \emph{irreversible} port-Hamiltonian system (see \cite{Ramirez2013} and \cite{Zitte2019} for details)
\begin{equation}\label{eq:HEX_IPHS}
\begin{aligned}
\dot{x} &= R\left(x,\frac{\partial S}{\partial x}(x)\right)J
\frac{\partial H}{\partial x}(x)
\\[0.5em]  & \qquad 
+ \underbrace{\begin{bmatrix} -H_c + H_c^\text{in} \\ 0 \end{bmatrix}}_{G_c} u_c 
+ \underbrace{\begin{bmatrix} 0\\ -H_h + H_h^\text{in} \end{bmatrix}}_{G_h} u_h,
\end{aligned}
\end{equation} 
with inputs $u_i := \dfrac{Q_i}{\rho_i V}$, $i \in \{c,h\}$, $J=\begin{bmatrix} 0& 1\\-1& 0 \end{bmatrix}$ and
\begin{equation*}
R\left(x,\frac{\partial S}{\partial x}(x)\right) = \lambda \left(\frac{1}{T_c} - \frac{1}{T_h} \right),
\end{equation*}
where $\lambda$ denotes the (constant) heat transfer coefficient. 

Now, since $J$ is skew-symmetric the overall system is cyclo-lossless, i.e.,
\begin{equation}
\frac{dH}{dt}(x) = y_c u_c  + y_h u_h,  
\end{equation}
where the corresponding natural outputs are given by
\begin{equation}\label{eq:HEX_outputs}
y_i = G_i^T(x) \frac{\partial H}{\partial x}(x) = - H_i + H_i^\text{in}, \ i \in \{c,h\}.
\end{equation} 

Because of the off-diagonal terms in the $J$-matrix, the heat exchanger (\ref{eq:HEX_IPHS})--(\ref{eq:HEX_outputs}) does \emph{not} satisfy the structural conditions of Theorem \ref{theorem}. However, keeping the cold stream enthalpy constant, i.e, ${H_c = \bar{H}_c}$, and assuming $H_i > 0$, reveals that the heat exchanger is one-port \emph{passive} from the perspective of the hot stream. Indeed, $\bar{T}_c < T_h$ yields $R(\cdot) > 0$, implying
\begin{equation}\label{eq:heat_ineq_h}
\frac{dH}{dt}(\bar{H}_c,H_h) \leq  y_hu_h.
\end{equation}  

Conversely, the heat exchanger is \emph{not} one-port passive from the perspective of the cold stream, which is evident from the thermal power balance
\begin{equation}\label{eq:heat_ineq_c}
\frac{dH}{dt}(H_c,\bar{H}_h) = \underbrace{\lambda \left(\frac{1}{T_c} - \frac{1}{\bar{T}_h} \right)}_{> 0} + \, y_c u_c, 
\end{equation} 
for all $\bar{T}_h > T_c$. 

The physical reason in both cases is rather obvious: under the condition that ${T_c < T_h}$, the cold stream will always experience an increase of energy from the transfer of (free) energy from the hot stream, but never vice--versa. This is the working principle of a heat exchanger. Of course, in the special case ${T_c = T_h}$, there is no energy exchange since $R(\cdot) = 0$, which results in both ports being one-port lossless.  

\section{CONCLUSIONS AND OUTLOOK}
\label{sec:conc}

The main part of this paper is devoted to the generalization of the fact (as implied by Kelvin's formulation of the Second Law) that thermodynamic systems cannot convert thermal energy into mechanical energy (i.e., work) in a recurrent manner, while maintaining the same temperature (i.e., by using a single heat source). This results in the notion of one-port cyclo-passivity for general port-Hamiltonian systems with multiple ports. Sufficient conditions for one-port cyclo-passivity are derived in Theorem \ref{theorem}, where it is additionally shown that the storage function for one-port cyclo-passivity can be directly computed as a partial Legendre transform of the Hamiltonian. Restricting to such storage functions the sufficient conditions of Theorem \ref{theorem} are close to necessary as well. 
Thus Theorem \ref{theorem} formalizes fundamental limitations of energy conversion for general physical systems with multiple ports. This is illustrated on a wide variety of energy-converting systems.

On the other hand, again drawing inspiration from thermodynamics, we know that thermal energy {\it can} be converted into mechanical energy by the use of at least two heat sources, as evidenced by the Carnot cycle. Thus the next question to be addressed is how in a similar way one-port cyclo-passive port-Hamiltonian systems {\it can} convert energy from one port to another in a recurrent manner, by using more than one set-point value for the output of the first port. One option is to mimic the Carnot cycle by combining Theorem \ref{theorem} (corresponding to {\it isothermal} processes) with Remark \ref{remark} (corresponding to {\it isentropic} (adiabatic) processes); see already the end of Subsection \ref{gas}. This naturally leads to the question of {\it maximal efficiency} of port-Hamiltonian systems with two ports. Another option, as exemplified by the synchronous motor case, is to use periodic set-point functions. These questions are left for future research.
%

\renewcommand\refname{REFERENCES}

\end{document}